\documentclass[accepted]{uai2021} 

 \usepackage[british]{babel}

\usepackage{amsmath, amsthm}
\usepackage[linesnumbered, ruled]{algorithm2e}

\usepackage{natbib} 

\usepackage{mathtools} 
\usepackage{booktabs} 
\usepackage{tikz} 
\usetikzlibrary{shapes}

\theoremstyle{plain}
\newtheorem{lem}{Lemma}[section]
\newtheorem{thm}[lem]{Theorem}
\newtheorem{prop}[lem]{Proposition}
\newtheorem{cor}[lem]{Corollary}

\theoremstyle{definition}
\newtheorem{dfn}[lem]{Definition}
\newtheorem{rmk}[lem]{Remark}
\newtheorem{exm}[lem]{Example}

\newcommand{\M}{{\mathcal{M}}}

\newcommand{\X}{{\mathcal{X}}}

\newcommand{\cH}{\mathcal{H}}
\newcommand{\G}{\mathcal{G}}

\newcommand{\cmid}{\,|\,}

\renewcommand{\tilde}{\widetilde}

\newcommand{\lrarr}{\leftrightarrow}

\DeclareMathOperator{\dis}{dis}
\DeclareMathOperator{\pa}{pa}
\DeclareMathOperator{\ch}{ch}
\DeclareMathOperator{\dec}{de}
\DeclareMathOperator{\an}{an}
\DeclareMathOperator{\sib}{sib}
\DeclareMathOperator{\mbl}{mb}

\newcommand\indep{\protect\mathpalette{\protect\independenT}{\perp}}
\def\independenT#1#2{\mathrel{\rlap{$#1#2$}\mkern2mu{#1#2}}}

\title{Dependency in DAG models with Hidden Variables.}

\author[1]{\href{mailto:Robin J. Evans <evans@stats.ox.ac.uk>?Subject=Your UAI 2021 paper}{Robin J.~Evans}{}} 
\affil[1]{%
    Department of Statistics\\
    University of Oxford\\
    Oxford\\
    United Kingdom
}

\begin{document}

\maketitle

\begin{abstract}
Directed acyclic graph models with hidden variables have been much studied, particularly in view of their computational efficiency and connection with causal methods.  In this paper we provide the circumstances under which it is possible for two variables to be identically equal, while all other observed variables stay jointly independent of them and mutually of each other.  We find that this is possible if and only if the two variables are `densely connected'; in other words, if applications of identifiable causal interventions on the graph cannot (non-trivially) separate them.  As a consequence of this, we can also allow such pairs of random variables have any bivariate joint distribution that we choose.  This has implications for model search, since it suggests that we can reduce to only consider graphs in which densely connected vertices are always joined by an edge.
\end{abstract}

\section{Introduction}\label{sec:intro}

Informally, a directed acyclic graph (DAG) is a collection of
vertices (or nodes) joined by directed edges ($\to$) such that there is 
no directed path from a vertex to itself.  DAGs can be used to 
describe \emph{Bayesian Networks} by allowing each vertex to
depend stochastically on its parent nodes.  

We may hypothesize the existence of vertices which we cannot see,
and these are referred to as \emph{hidden} or \emph{latent} nodes
or variables.

\begin{exm} \label{exm:iv}
Consider the instrumental variables model in Figure \ref{fig:iv}(a).  In this 
case we have a DAG over four variables, of which one ($h$) is hidden.  Suppose 
all the variables are binary, and that $H$ (the random variable associated with $h$) is Bernoulli with probability $1/2$; 
we select some value for $X_a$.  Then define each of $X_b$ and $X_c$ to be an
`xor' gate of their two parents; we denote this addition modulo 2 using the symbol 
$\oplus$.  Then we find that:
\begin{align*}
X_b &= X_a \oplus H \\ 
X_c &= X_b \oplus H = (X_a \oplus H) \oplus H = X_a.
\end{align*}
Hence, $X_a = X_c$ with probability 1, but $X_b$ involves $H$, and is therefore
(marginally) independent of $X_a$ and $X_c$.  This is shown in Table 
\ref{tab:truth}, where the marginal distributions of $X_a,X_b$ and of
$X_b,X_c$ are just the Cartesian products $\{0,1\}^2$, but $X_a = X_c$ for 
all values.

\begin{table}
\begin{center}
\begin{tabular}{c|c||c|c}
$H$ & $X_a$ & $X_b$ & $X_c$\\
\hline
0 & 0 & 0 & 0\\
1 & 0 & 1 & 0\\
0 & 1 & 1 & 1\\
1 & 1 & 0 & 1\\
\end{tabular}
\end{center}
\caption{Table showing the possible values for $X_b = X_a \oplus H$ and $X_c = X_b \oplus H$ 
given combinations of $X_a$ and $H$.}
\label{tab:truth}
\end{table}

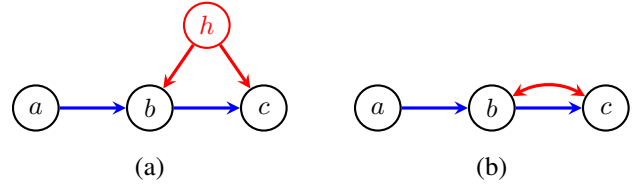
\begin{figure}
 \begin{center}
 \begin{tikzpicture}
[rv/.style={circle, draw, thick, minimum size=6mm, inner sep=0.5mm}, node distance=15mm, >=stealth]
 \pgfsetarrows{latex-latex};
 \begin{scope}
 \node[rv]  (1)              {$a$};
 \node[rv, right of=1] (2) {$b$};
 \node[rv, right of=2] (3) {$c$};
 \node[rv, color=red, above of=3, xshift=-7.5mm,yshift=-4mm] (4) {$h$};
 \draw[->, very thick, color=blue] (1) -- (2);
 \draw[->, very thick, color=blue] (2) -- (3);
 \draw[<-, very thick, color=red] (2) -- (4);
 \draw[<-, very thick, color=red] (3) -- (4);
 \node[below of=2, xshift=0mm, yshift=7mm] {(a)};
 \end{scope}
 \begin{scope}[xshift=4.5cm]
 \node[rv]  (1)              {$a$};
 \node[rv, right of=1] (2) {$b$};
 \node[rv, right of=2] (3) {$c$};
 \draw[->, very thick, color=blue] (1) -- (2);
 \draw[->, very thick, color=blue] (2) -- (3);
 \draw[<->, very thick, color=red] (2) to[bend left] (3);
 \node[below of=2, xshift=0mm, yshift=7mm] {(b)};
 \end{scope}
 \end{tikzpicture}
 \caption{The \emph{instrumental variables} model. 
 (a) A directed acyclic graph on four vertices, one of which is unobserved.  
 (b) The latent projection of this graph to an ADMG over the observed vertices.}
 \label{fig:iv}
 \end{center}
\end{figure}
\end{exm}

This is a surprising result, since $X_a$ and $X_c$ do not share any edge
nor a latent parent, so we might expect such a distribution not to be achievable.  
This can be easily extended to a sequence of bits of arbitrary cardinality, 
just by concatenating the values and applying the arithmetic in a bitwise 
fashion. 
In this paper we will give a generalization of this result to precisely 
characterize when it is possible for two variables to be equal, 
and independent of all other observed variables (which themselves will all
be mutually independent).  We will show that it is 
possible if and only if there is no `nested' 
Markov constraint between the two variables in question. 

\subsection{Previous Work and Contribution}

The question of which distributions can be realized from a 
directed acyclic graph model with hidden variables has been a topic
of considerable interest in recent years.  Its roots lie a century back in 
the structural diagrams of \citet{wright21correlation}, later formalized 
as Bayesian Networks by \citet{pearl86fusion}, and their Markov properties
derived by Dawid, Geiger, Pearl, Lauritzen, Verma and others.  Hidden variables 
were introduced in the 1980's, and in this context Pearl and 
Verma derived the so-called `Verma constraint', a constraint that
is similar in form to, but distinct from, conditional independence; this 
restriction had been noted earlier by \citet{robins86new}.  This was 
later turned into an algorithm by \citet{tian02on} and a \emph{nested Markov} model by 
\citet{richardson:17}, proven algebraically complete in the discrete 
case by \citet{evans:18}.


In this paper we show a different kind of completeness result for
the nested Markov property to that given by \citeauthor{evans:18}: 
approximately, we show that two vertices are 
`densely connected' (see Definition \ref{dfn:dc}), if and only if their 
associated variables can be 
made to be equal to one another, while remaining independent of all other observed 
variables.  In fact, we can always reduce the nested Markov property to 
a \emph{maximal arid graph} (MArG) in the same nested Markov equivalence class.  
This graph has edges precisely between densely
connected vertices in the original graph, and advertises a nested constraint 
between every non-adjacent pair.

There are three reasons why this result is interesting.  First, it gives a 
quick way of checking whether a particular class of distributions is compatible 
with a hidden variable model; this is a topic of considerable interest, 
particularly among computer scientists
\citep[see, e.g.][]{wolfe19inflation, navascues20inflation}.
Second, it reinforces the importance of the nested Markov property, because 
it shows that it precisely characterizes when two variables can be made to 
have an arbitrary joint distribution, independent of all other (observed) variables.  
The rest of the paper is organized as follows.  Finally, it suggests that we
ought to restrict the class of structures that we consider in a causal model
search to maximal arid graphs.  

The rest of the paper is organized as follows.  Section \ref{sec:prelim} gives 
preliminary definitions, Section \ref{sec:nested} discusses the models and
Section \ref{sec:arid} introduces arid graphs. Then Sections \ref{sec:percor} and 
\ref{sec:main} give our main results.  Section \ref{sec:mini} discusses
finding minimal graphs, and Section \ref{sec:exm} gives some worked examples.  
We end with a discussion.
Code to replicate many of the analyses is contained in the R package 
\texttt{dependence} \citep{dependence:21}.

\section{Preliminaries}\label{sec:prelim}

In this paper we consider two classes of graph: directed acyclic graphs (DAGs), and 
acyclic directed mixed graphs (ADMGs).  Directed acyclic graphs have a set of 
vertices $V$, and a collection of directed ($\to$) edges between pairs of distinct 
vertices.  The only condition in these edges is that one cannot follow a path that
coheres with the direction of the edges and end up back where one started. 
An ADMG is a DAG together with a collection of unordered pairs of bidirected 
($\lrarr$) edges.  These (informally) represent hidden variables, and so are used
in circumstances where one does not believe that all variables have been observed. \
Note, therefore, that a DAG is an ADMG, but not the other way around in general. 
Examples of DAGs are given in Figures \ref{fig:iv}(a) and \ref{fig:gadget}(a), and 
of ADMGs in Figures \ref{fig:iv}(b), \ref{fig:gadget}(b), \ref{fig:counter_ex} and 
\ref{fig:counter_exm}.

Any pair of vertices connected
by an edge is said to be \emph{adjacent}.  A pair of vertices may be connected
by either a bidirected edge or a directed edge, or both; the latter configuration
is called a \emph{bow} (see $b$ and $c$ in Figure \ref{fig:iv}(b), 
for example).

A \emph{path} is a sequence of incident edges between adjacent, distinct 
vertices in a graph; note that two distinct paths in a mixed graph may have 
the same vertex set if a bow is present. 
The path is \emph{directed} if all the edges are directed, and oriented to 
point away from the first vertex towards the last. 



\subsection{Familial Definitions}

We will use standard genealogical terminology for relations between vertices.  
Given a vertex $v$ in an ADMG $\G$ with a vertex set $V$, define
the sets of \emph{parents}, 
\emph{children}, 
\emph{ancestors}, 
and \emph{descendants}
of $v$ as
\begin{align*}
\pa_\G(v) &\equiv \{ w : w \to v \text{ in }\G \}\\
\ch_\G(v) &\equiv \{ w : v \to w \text{ in }\G \}\\
\an_\G(v) &\equiv \{ w : w=v \text{ or } w \to \cdots \to v\text{ in }\G \}\\
\dec_\G(v) &\equiv \{ w : w=v \text{ or } v \to \cdots \to w\text{ in }\G \},
\end{align*}
respectively.  
These definitions also apply disjunctively to sets, e.g.\ for a set of 
vertices $C \subseteq V$, $\pa_\G(C) \equiv \bigcup_{v \in C} \pa_\G(v)$.
%
In addition, we define the \emph{district} of $v$ to be the set of vertices
reachable by paths of bidirected edges:
\[
\dis_\G(v) \equiv \{v\} \cup \{ w : w \lrarr \ldots \lrarr v\text{ in }\G \}.
\]
The set of districts of an ADMG $\G$
always partitions the set of vertices $V$.  If the graph is 
unambiguous, we may omit the subscript $\G$: e.g.\ $X_{\pa(v)}$.


Given an ADMG $\G$, and a subset $S$ of vertices $V$ in $\G$, the 
\emph{induced subgraph} $\G_S$ is the graph
with vertex set $S$, and those edges in $\G$ between elements in $S$.  
A set $S \subseteq V$ is called \emph{bidirected-connected} 
in $\G$ if every vertex in $S$ can be reached from every other 
using only a bidirected path that is contained entirely in $S$. 

We sometimes abbreviate (e.g.) $\an_{\G_S}(v)$ and $\dis_{\G_S}(v)$ to 
$\an_S(v)$ and $\dis_S(v)$ where the graph is clear.

\section{Markov Models} \label{sec:nested}

We consider random variables $X_V \equiv (X_v : v \in V)$ taking values in the product space 
$\X_V = \times_{v \in V} \X_v$, for finite dimensional sets $\X_v$.  
For any $A \subseteq V$ we denote the subset $(X_v : v \in A)$ 
by $X_A$.  For a particular value $x_V \in \X_V$ we similarly denote
a subvector over the set $A$ by $x_A$.

We will say that a distribution $p$ is \emph{Markov} with respect to a 
DAG $\G$, if for each $v \in V$, we can write 
\begin{align*}
X_v = f_v(X_{\pa_\G(v)}, E_v)
\end{align*}
for some measurable function $f_v$, and independent noise variables
$E_v$.  That is, each variable depends on its predecessors only through 
the value of its direct parents in the graph.  If the resulting joint 
distribution admits a density $p$, then
this implies the usual factorization:
\begin{align*}
p(x_V) &= \prod_{v \in V} p(x_v \mid x_{\pa(v)}), && x_V \in \X_V.
\end{align*}
This corresponds to a set of conditional independences over the 
space $\X_V$.  See \citet{lauritzen90independence} for more details.

\subsection{Canonical DAGs and Latent Projection}

Given an ADMG $\G$, we define the \emph{canonical DAG}, $\overline{\G}$ 
as the DAG in which each bidirected edge is replaced by a hidden 
variable with exactly two children, being the same vertices that were 
the endpoints of the bidirected edge.  Note that now all bidirected 
edges have been replaced with two directed edges; since the directed 
part of the ADMG is acyclic, then certainly the resulting 
graph is acyclic.  Hence, the canonical DAG is---as its 
name implies---always a DAG.  For example, the DAG in Figure \ref{fig:iv}(a)
is the canonical DAG for the graph in Figure \ref{fig:iv}(b), and 
similarly Figure \ref{fig:gadget}(a) for \ref{fig:gadget}(b).

Given a DAG, we can transform it into an ADMG that captures most of
the causal structure by performing a \emph{latent projection}.  Simple
examples correspond to the pairs of graphs in Figures \ref{fig:iv} and 
\ref{fig:gadget}, but we provide a definition and another example 
in the Appendix, Section \ref{sec:lp}.

Given an ADMG $\G$ with vertices $V$, we define the \emph{marginal 
model} as the set of distributions which can be realized as a margin
over $X_V$ for distributions that are Markov with respect to the 
canonical DAG $\overline{\G}$.  In doing this, we make no assumption 
about the statespace of the hidden variables, though for discrete
models it is sufficient to have the same cardinality as all the 
observed variables, and in general it is both necessary and 
sufficient for the latent variables we use to be continuous. 

We now define sets which are \emph{intrinsic} for a particular 
ADMG.  Take a set $B \subseteq V$, and set $B^{(0)} = V$; 
then alternately apply the operations:
\begin{align*}
B^{(i+1)} = \dis_{B^{(i)}}(B) \qquad B^{(i+2)} = \an_{B^{(i+1)}}(B),
\end{align*}
increasing $i$ at each step. 
Each time we apply these steps either at least one vertex will be removed
from $B^{(i)}$, or else the process will terminate at a superset of $B$, 
which we will denote by $\langle B \rangle_\G$.  This is called the 
\emph{closure} of $B$.
If $\langle B \rangle_\G$ is 
bidirected-connected, then we say it is an \emph{intrinsic} set, 
and the \emph{intrinsic closure} of $B$. 

For convenience we generally abbreviate $\langle \{v\} \rangle_\G$ as 
$\langle v \rangle_\G$, and (for example) $\pa_\G(\langle \{v\} \rangle_\G)$
as $\pa_\G(\langle v \rangle)$.

%


\section{Arid Graphs} \label{sec:arid}

The main result of this section is taken from \citet{shpitser:18}, and says 
that the nested Markov model
associated with \emph{any} ADMG $\G$ can be associated, without
loss of generality, with a closely related \emph{maximal arid graph (MArG)}
$\G^{\dag}$.  In particular, the nested Markov models associated
with $\G$ and $\G^\dag$ are the same.  Note that MArGs are themselves 
just a restricted class of ADMGs.
 

\subsection{Arid Graphs} \label{sec:arid2}

%
%

\begin{dfn} \label{dfn:arid}
%
Let $\G$ be an ADMG.  We say that $\G$ is \emph{arid} if 
$\langle v\rangle_\G = \{v\}$ for each $v \in V$.
\end{dfn}

The word `arid' is used because it implies that the graph 
lacks any (non-trivial) `C-trees' or `converging aborescences'.

\begin{dfn} \label{dfn:dc}
A pair of vertices $a \neq b$ in an ADMG $\G$ is \emph{densely connected} if
either $a \in \pa_\G(\langle b \rangle)$, or
$b \in \pa_\G(\langle a \rangle)$, or
$\langle \{ a, b \} \rangle_\G$ is a bidirected-connected set.

An ADMG $\G$ is called \emph{maximal} if every pair of densely connected 
vertices in $\G$ are adjacent. 
\end{dfn}

Densely connected pairs of vertices form the nested Markov analogue of 
\emph{inducing paths} \citep{verma90equiv}.  The existence of an
inducing path between two vertices means that (almost) no distribution
that is Markov with respect to the graph will have \emph{any} conditional 
independence between the associated variables.  
Analogously, a densely connected pair of vertices means that 
(almost) no distributions that are nested Markov with respect to the graph
will have \emph{any} conditional independences within \emph{any} ADMG 
corresponding to a valid combination of intrinsic sets. 
%
In effect, a densely connected pair cannot be made independent, by any
combination of conditioning and identifiable intervention operations.  

As an example, note that the pair $\{a,c\}$ is densely connected 
in Figure \ref{fig:iv}(b), because $\pa_\G(\langle c\rangle) = \pa_\G(\{b,c\}) = \{a,b\}$. 
Further details are given in the Appendix, Section \ref{sec:nestedmodel}.

\begin{dfn} \label{dfn:aridproj}
Let $\G$ be an ADMG.  We define its 
\emph{maximal arid} projection as 
$\G^{\dag}$, the ADMG 
with edges:
\begin{itemize}
\item $a \to b$ if and only if $a \in \pa_\G(\langle b \rangle)$;
\item $a \lrarr b$ if and only if the set $\langle \{a,b\} \rangle_\G$ is 
bidirected-connected, and both $a \notin \pa_{\G^{\dag}}(b)$ and 
$b \notin \pa_{\G^{\dag}}(a)$; that is, there is no directed edge between 
$a$ and $b$.
\end{itemize}
\end{dfn}

Note that, in particular, all directed edges in $\G$ are preserved 
in $\G^\dag$, though some may be added if they preserve ancestral sets. 
Bidirected edges will be removed if connecting a vertex to something in 
its intrinsic closure, and added if they are required for maximality. 
For example, the MArG projection of the graph in Figure \ref{fig:gadget}(b)
adds a bidirected edge between the vertices 3 and 4.
A further example is found in the Appendix, Section \ref{sec:aridproj}.

%

\begin{prop}[\citealp{shpitser:18}]
The maximal arid projection is arid and maximal.  
\end{prop}

\begin{thm}[\citealp{shpitser:18}]
The nested model associated with an ADMG $\G$ is the same
as the nested model associated with $\G^{\dag}$.
\end{thm}

\begin{cor}
Let $p$ be a distribution that is nested Markov with
respect to $\G$.  Then for $v,w \in V$ the variables $X_v$ and
$X_w$ have a nested constraint between them if 
and only if $v$ and $w$ are not adjacent in the arid projection 
$\G^\dag$. 
\end{cor}

\section{Perfect Correlation} \label{sec:percor}

We now come to the main content of this paper. 

We have already seen
in Example \ref{exm:iv} that it is possible for two vertices which are 
neither joined by an edge, nor share a latent parent, nevertheless to 
be perfectly correlated. 
We now give a slightly more complicated example of this phenomenon.

\begin{exm}
Consider the DAG in Figure \ref{fig:gadget}(a).  Suppose that each hidden variable 
($H_1, H_2, H_3$) is again a Bernoulli random variable with probability $1/2$, and all the 
observed variables are again `xor' gates.  Then:
\begin{align*}
X_a &= H_1 \oplus H_2  & X_b &= H_1 \oplus H_3\\
X_c &= X_b \oplus H_2  & X_d &= X_a \oplus H_3,
\end{align*}
and hence $X_c = X_d = H_1 \oplus H_2 \oplus H_3$.  It follows from 
Lemma \ref{lem:sumtozero2} (in Appendix \ref{sec:proofs})
that if $H_1,H_2,H_3$ are all independent, then so are $X_a,X_b,X_c$.
Naturally, the same result also holds for $X_a,X_b,X_d$. 
\end{exm}

\begin{figure}
 \begin{center}
 \begin{tikzpicture}
[rv/.style={circle, draw, thick, minimum size=6mm, inner sep=0.5mm}, node distance=20mm, >=stealth]
 \pgfsetarrows{latex-latex};
 \large
 \begin{scope}
 \node[rv]  (1)              {$a$};
 \node[rv, right of=1] (2) {$b$};
 \node[rv, below of=1] (3) {$c$};
 \node[rv, right of=3] (4) {$d$};
 \node[rv, right of=1, color=red, xshift=-10mm, yshift=5mm]  (U12)              {\normalsize $h_1$};
 \node[rv, below of=1, color=red, yshift=10mm, xshift=-5mm]  (U13)              {\normalsize $h_2$};
 \node[rv, below of=2, color=red, yshift=10mm, xshift=5mm]  (U24)              {\normalsize $h_3$};
 \draw[->, very thick, color=blue] (1) -- (4);
 \draw[->, very thick, color=blue] (2) -- (3);
 \draw[<-, very thick, color=red] (2) -- (U12);
 \draw[<-, very thick, color=red] (1) -- (U12);
 \draw[->, very thick, color=red] (U13) -- (1);
 \draw[->, very thick, color=red] (U13) -- (3);
 \draw[->, very thick, color=red] (U24) -- (2);
 \draw[->, very thick, color=red] (U24) -- (4);
 \node[below of=3, xshift=10mm, yshift=12mm] {(a)};
 \end{scope}
 \begin{scope}[xshift=4.5cm]
 \node[rv]  (1)              {$a$};
 \node[rv, right of=1] (2) {$b$};
 \node[rv, below of=1] (3) {$c$};
 \node[rv, right of=3] (4) {$d$};
 \draw[->, very thick, color=blue] (1) -- (4);
 \draw[->, very thick, color=blue] (2) -- (3);
 \draw[<->, very thick, color=red] (2) -- (1);
 \draw[<->, very thick, color=red] (1) -- (3);
 \draw[<->, very thick, color=red] (4) -- (2);
 \node[below of=3, xshift=10mm, yshift=12mm] {(b)};
  \end{scope}
 \normalsize
 \end{tikzpicture}
 \caption{(a) A directed acyclic graph on seven vertices, three of which are unobserved.  
 (b) The latent projection of this graph to an ADMG over the observed vertices.}
 \label{fig:gadget}
 \end{center}
\end{figure}
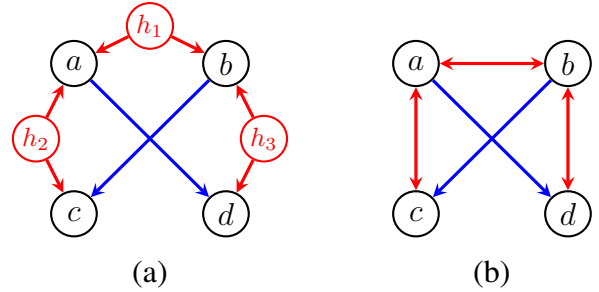

Again, the result seems surprising given the lack of any 
form of adjacency between the vertices $c$ and $d$.

\subsection{Minimal Closures}

In our main result (Theorem \ref{thm:main}) we will claim that it is possible
to have a distribution in the marginal model of an ADMG, with two variables 
identical to one another and independent of all other observable variables if 
and only if they are densely connected.  If these vertices are
adjacent in the ADMG, then the result is obvious; if we have $w \to v$ then the 
value can be directly transmitted, and if $w \lrarr v$ then the latent parent 
in the canonical DAG can just give the same value to both $v$ and 
$w$.  We will show that, in fact, this can be applied much more generally.

\begin{prop} \label{prop:trees}
Let $\G$ be an ADMG, and $B = \{v_1, \ldots, v_k\} \subseteq V$.  Then 
assume that $C := \langle B \rangle_\G$ is a 
bidirected-connected (and hence intrinsic) set. 

Now consider the following edge subgraph of $\G_C$, which we 
call $\tilde\G_{C}$: 
the directed edges only contain a forest over 
$C$ that converges on $B$, and 
bidirected edges only contain a spanning tree over $C$.
Then $\langle B \rangle_\G  = C = \langle B \rangle_{\tilde\G_{C}}$.
%
\end{prop}

\begin{proof}
Since every element of $\langle B \rangle_\G$ is an ancestor of 
some $v \in B$ in $\tilde\G_C$, and also in the 
same district as everything in $B$, it is clear that the result holds. 
\end{proof}

Note that this loosely defined procedure 
will generally lead to many different graphs, but they will
always have the same total number of edges; that is, there will 
always be $|C| - 1$ bidirected edges, and 
$|C| - |B|$ directed edges.  However, the 
particular choice of the trees will affect which edges can be 
removed by the next result.  In the Appendix (Section \ref{sec:algo}) 
we give some algorithms for formally determining how to
pick the trees in an efficient way.

Suppose we have a graph whose bidirected edges form a spanning tree, and
a set of vertices $D \subseteq V$ that divides into $d$ 
districts.  We say that $D$ 
is \emph{almost encapsulated} in $\G$ if for every bidircted edge 
with an endpoint in $D$, the other endpoint is also in $D$, with 
the exception of exactly $d$ edges, one for each connected component.



\begin{prop} \label{prop:complete}
Consider an ADMG $\cH = \tilde\G_C$ obtained by application of Proposition 
\ref{prop:trees}, and let $B = C \setminus \pa_\cH(C)$ be the vertices
without children.  Suppose that there is a subset of vertices
$A \subseteq C \setminus B$ such that $\an_\cH(A)$ is almost
encapsulated.  Then if we remove $\an_\cH(A)$ from the graph 
to obtain (say) $\tilde{\cH}$ we have 
$\langle B \rangle_{\tilde\cH} = \langle B \rangle_\cH \setminus \an_\cH(A)$.
\end{prop}

\begin{proof}
Since the set we remove is almost encapsulated, this means 
that the remaining vertices still have a spanning tree of 
bidirected edges.  In addition, since the subset is ancestral, 
every remaining vertex is still an ancestor of an element of
$B$.  Hence we have the result given.
\end{proof}

This is a very useful result, since it will enable us to find
a minimal graph which we can apply our results to.  However, finding
that set may be computationally difficult, as the next example 
illustrates.

\begin{exm} \label{exm:mini}
Consider the graph in Figure \ref{fig:counter_ex}, and suppose we wish
to have $X_v = X_w$ and all other variables independent.  This is clearly 
very easy, since there is a directed edge $w \to v$; however, we cannot 
reach this minimal graph all that easily by using Proposition \ref{prop:complete}. 
No set will work except for $\{c,e\}$, and we can make this problem have
exponential complexity very easily (see Section \ref{sec:diff}).
%
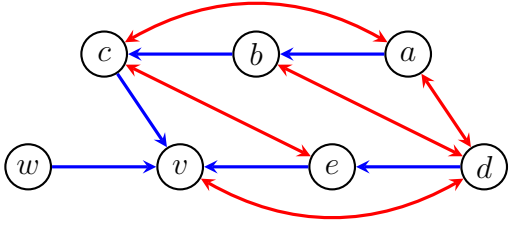
\begin{figure}
 \begin{center}
 \begin{tikzpicture}
[rv/.style={circle, draw, thick, minimum size=6mm, inner sep=0.5mm}, node distance=20mm, >=stealth]
 \pgfsetarrows{latex-latex};
 \large
 \begin{scope}
 \node[rv]  (1)              {$w$};
 \node[rv, right of=1] (2) {$v$};
 \node[rv, right of=2] (7) {$e$};
 \node[rv, right of=7] (6) {$d$};
 \node[rv, above of=1, yshift=-5mm, xshift=10mm] (3) {$c$};
 \node[rv, right of=3] (4) {$b$};
 \node[rv, right of=4] (5) {$a$};
 \draw[->, very thick, color=blue] (5) -- (4);
 \draw[->, very thick, color=blue] (4) -- (3);
 \draw[->, very thick, color=blue] (3) -- (2);
 \draw[->, very thick, color=blue] (1) -- (2);
 \draw[->, very thick, color=blue] (7) -- (2);
 \draw[->, very thick, color=blue] (6) -- (7);
 \draw[<->, very thick, color=red] (3) -- (7);
 \draw[<->, very thick, color=red] (5) -- (6);
 \draw[<->, very thick, color=red] (4) -- (6);
 \draw[<->, very thick, color=red] (5) to[bend right] (3);
 \draw[<->, very thick, color=red] (2) to[bend right] (6);
 \end{scope}
 \normalsize
 \end{tikzpicture}
 \caption{An ADMG which is not minimal w.r.t.\ $(v,w)$ but for which we cannot
 apply Corollary \ref{cor:trees3}.  See Example \ref{exm:mini}.}
 \label{fig:counter_ex}
 \end{center}
\end{figure}

As a result of this problem, we present two much easier to 
apply corollaries.
\end{exm}

\begin{cor} \label{cor:trees3}
Let $\cH = \tilde\G_C$ be an ADMG chosen 
using Proposition \ref{prop:trees}, and let $B = C \setminus \pa_{\cH}(C)$.  
Suppose there is a 
vertex $a$ such that $\ch_{\cH}(a) \neq \emptyset$, and such 
that $\an_{\cH}(a)$ is almost encapsulated within $\cH$.  

Then we can remove $\an_{\cH}(a)$ from the graph to obtain (say)
$\tilde{\cH}$, and still have
$C \setminus \an_{\cH}(a) = \langle B \rangle_{\tilde{\cH}}$.
\end{cor}

\begin{proof}
This is just a special case of Proposition \ref{prop:complete} 
with a single vertex $a$ instead of a set $A$.
\end{proof}


In a graph we say that a vertex is a \emph{leaf} if it only has
a single neighbour. 

\begin{cor} \label{cor:trees2}
Let $\cH = \tilde\G_C$ be an ADMG chosen 
using Proposition \ref{prop:trees}, and let $B = C \setminus \pa_{\cH}(C)$.  
If some $a \in C \setminus B$ is 
a leaf in both the directed and the bidirected trees, then we can remove 
it from the graph and still have 
$C \setminus \{a\} = \langle B \rangle_{\cH_{-a}}$.
\end{cor}

\begin{proof}
This is just a special case of Corollary \ref{cor:trees3} 
with a single vertex acting as the set of ancestors.
\end{proof}

%

The complexity of identifying a minimal set using Proposition 
\ref{prop:complete} motivates us to describe a linear time algorithm
for determining the an inclusion minimal subgraph in Section \ref{sec:mini} 
(see Algorithm \ref{alg:mini1}).



\section{Main Results} \label{sec:main}

In this section we deal with our main concern, which is to show
that if two vertices ($v$ and $w$) are densely connected, then 
we can set 
them to be equal, provided that all the other variables we 
need to set have at least the same cardinality as $X_v$ 
and $X_w$.

\begin{dfn} \label{dfn:Gvw}
Given a graph $\G$ and two vertices $v,w \in V$ that are densely connected, we
define the subgraph $\G^{vw}$ as follows:

\begin{itemize}[leftmargin=*]
\item if $w \in \pa_\G(\langle v\rangle)$ then use the induced
subgraph $\G_{\langle v\rangle \cup \{w\}}$;
\item if $v \in \pa_\G(\langle w\rangle)$ then take $\G_{\langle w\rangle \cup \{v\}}$;
\item if $\langle \{v,w\} \rangle_\G$ is bidirected-connected, then take the 
induced subgraph $\G_{\langle \{v,w\} \rangle}$.
\end{itemize}
\end{dfn}

Note that it is possible that \emph{both} $w \in \pa_\G(\langle v\rangle)$ 
\emph{and} $\langle \{v,w\} \rangle_\G$ is bidirected-connected, in 
which case we have a choice about which graph to select.  Typically 
we would prefer $\G_{\langle v\rangle \cup \{w\}}$, but it may be 
that $\G_{\langle \{v,w\} \rangle}$ enables us to select different 
variables to be used in the final distribution.  Such
greater flexibility could be useful in some circumstances (e.g.\ if
one of the variables used in the directed case does not have 
sufficiently high cardinality).  

\begin{prop} \label{prop:dir}
Let $\G$ be a ADMG with $v,w \in V$ such that
$w \in \pa_\G(\langle v \rangle)$.  Then there is a distribution,
Markov with respect to the canonical DAG $\overline\G$, such that
$X_v = X_w$ and these are independent of all other observed 
variables.  These other variables may be chosen to be 
mutually independent.
\end{prop}

\begin{proof}
We first assume that all variables are binary, and prove this
for a single bit. The extension to multiple bits 
is obvious, provided that all the variables have sufficiently
large cardinality.  We will ignore any vertices other than those in 
$\langle v\rangle_\G \cup \{w\}$, since these are just set to be independent
of all other random variables.  

Reduce the graph to $\G^{vw}$ and then apply 
Proposition \ref{prop:trees}; we can then choose to remove any 
unnecessary variables from $\langle v\rangle_\G$ using Proposition 
\ref{prop:complete} if we want to\footnote{Applying this 
result isn't strictly necessary, but 
it will reduce the number of variables that need to have 
the minimal cardinality.}. 


In this new smaller graph, say $\cH$, consider the canonical
DAG $\overline{\cH}$.  Set all hidden variables 
in $\overline{\cH}$ to be independent Bernoulli r.v.s with 
parameter $1/2$, and select some arbitrary $x_w \in \{0,1\}$.  
Now, let every other observed variable, including $X_v$, be the sum of its parents modulo 2.  
There is a unique directed path that carries the value of $X_w$ 
down to $X_v$ (though, of course, it will generally be encoded with various 
additions).  

Every hidden variable has exactly two children, 
and since these children are both ancestors of $v$ by exactly one 
path, this means that their value will have been cancelled out at
the node where the two paths first meet.  Consequently, the value 
of every bidirected edge will have been removed by the time we
reach $v$, and so $X_v = x_w$ as required.  

Now ignore $X_w$, and consider some subset $D \subseteq C$ 
of the other variables in the intrinsic closure of $v$.  Remember that there is 
a spanning tree
of bidirected edges, so in order to obtain that a set of variables is jointly
dependent we always need \emph{both} the endpoints of any bidirected edge. 
This clearly implies that we need the whole tree, and hence all of $C$; 
but recall that (exactly) one variable
has $X_w$ as a parent, so therefore even the entire set $C$ consists of 
independent random variables.
%
%
%
%
\end{proof}

The next proposition deals with the other way in which two vertices 
can be densely connected.

\begin{prop} \label{prop:bi}
Let $\G$ be an ADMG with $v,w \in V$ such that $\langle \{v,w\} \rangle_\G$ 
is bidirected-connected.  Then there is a distribution, Markov with 
respect to the canonical DAG $\overline\G$, such that $X_v = X_w$ and these 
are independent of all other observed variables.  Again, these variables
may be chosen to be mutually independent.
\end{prop}

\begin{proof}
First, obtain $\G^{vw}$ from Definition \ref{dfn:Gvw}, and then
find $\cH = \tilde\G^{vw}$ from Proposition \ref{prop:trees}.  
Then again (optionally) choose $\langle \{v,w\} \rangle_\cH$ to be minimal 
by applying Proposition \ref{prop:complete}.  
Then set all latent variables in $\overline\cH$ to be independent Bernoulli
random variables with parameter $1/2$.  

Set all observed variables to be the sum modulo 2 of their parents.  
Now, each latent variable has two children, and by minimality of the 
directed edges, each of 
these is an ancestor of exactly one of $v$ or $w$.  If both are ancestors 
of $v$ (or of $w$), then the value of the latent variable will cancel out at the 
vertex where the two descending paths meet.  Otherwise, 
the latent variable appears in the sums for both $X_v$ and $X_w$.
Hence $X_v = X_w$.  

For the other variables, any subset that (possibly) includes $v$ will 
involve only one end of at least one bidirected edge.  Hence by Lemma 
\ref{lem:sumtozero2} these variables are all independent. 
\end{proof}

Note that if neither of these conditions of Propositions \ref{prop:dir} 
or \ref{prop:bi} hold, then there must be a (nonparametric) nested 
constraint between $X_v$ and $X_w$ \citep{shpitser:18}. 

\begin{thm} \label{thm:main}
We can obtain any joint distribution on $(X_v,X_w)$ 
independently of all other variables if and only if in the 
MArG projection there is an edge between $v$ and $w$. 
\end{thm}

\begin{proof}
This follows from the two propositions above and the fact
that if there is no edge in the MArG projection, then there
is a nested constraint between $X_v$ and $X_w$ \citep{shpitser:18}. 
This amounts to a marginal independence constraint when other 
variables are mutually independent, so it is clearly incompatible 
with the distribution described.
\end{proof}

\begin{rmk}
We can easily extend the result to arbitrary discrete random variables
by using essentially the same method; assume that all the variables have 
exactly $k$ states, and set bidirected edges to be uniformly sampled from 
$\{0,1,\ldots, k-1\}$.  If both ends of the bidirected edge are ancestors 
of the same vertex (e.g.\ $v$) then one end of the bidirected edge must 
subtract its value (this way that quantity will still cancel out when the
two paths meet again).  If the 
edges are ancestors of distinct vertices then both children must 
add the number to their total. 
\end{rmk}

\subsection{Arbitrary Continuous Distributions} \label{sec:arb}

As we have already discussed, it is also easy to extend the result to 
continuous random variables by using bitwise operations; suppose that
the equation to determine $X_v$ is 
$X_v = \bigoplus_{i \in \pa_{\overline{\cal H}}(v)}Y_{i}$; 
here $Y_i$ can represent either an observed or a latent variable, and 
we have just proven that the right hand side is the same as $X_w$.  
To obtain an arbitrary joint distribution we can just replace the 
structural equation with $X_v = f(X_w, U)$ for an 
independent uniform $(0,1)$ random variable $U$, and a suitable measurable
function $f$ \citep[][Theorem 2.2]{chentsov:82}.
We illustrate this with the scatter plots in Figure \ref{fig:scatter},
which shows data generated from the IV model in this manner, with the 
marginal distributions set to be standard normal and the correlation 
between $X_a$ and $X_c$ set to be 0.9.

\begin{figure}
\includegraphics[width=\columnwidth, height=.9\columnwidth]{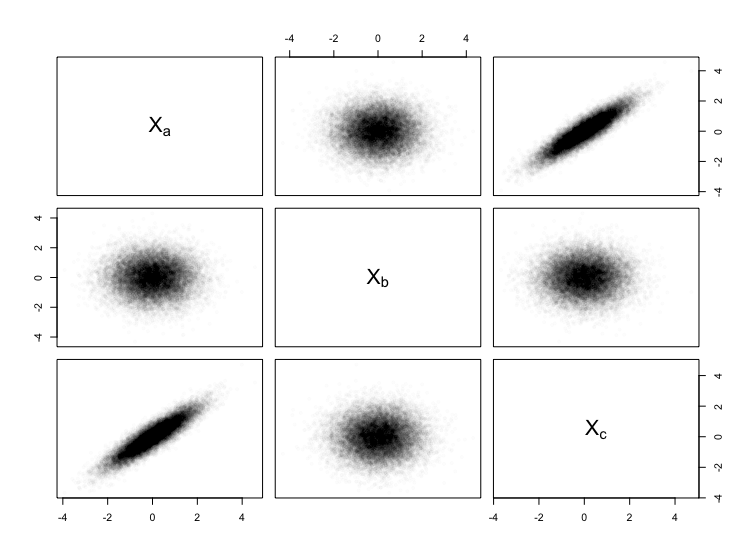}
\caption{Scatter plot showing normal random variables 
simulated from the hidden variable DAG in Figure \ref{fig:iv}.}
\label{fig:scatter}
\end{figure}

\section{Algorithms for Minimality} \label{sec:mini}

If we have applied Proposition \ref{prop:complete} as much as 
possible, then the graph we end up with is minimal.  However, the
result is not computationally efficient to use, since (in principal)
we may have to try every ancestral subset of $C \setminus B$.  In Algorithm 
\ref{alg:mini1} we give a fast algorithm 
to obtain a minimal set, by building it up from the two vertices $v,w$.  
In particular, the algorithm has worst-case linear complexity in 
the number of vertices. 

Recall that the bidirected edges in a graph output from Proposition
\ref{prop:trees} are in the form of a tree, so there is a unique path 
between any two vertices.  We denote by $\dis^v(A)$ the set of vertices 
on the unique bidirected paths from each $a \in A$ to $v$ 
(inclusive of $A$ and $v$).

\begin{algorithm} 
\SetAlgoLined
\KwIn{vertices $v,w$, graph $\G$ from Proposition \ref{prop:trees}}
\KwResult{A minimal set $W$.}
\eIf{$w \in \dis_\G(v)$}{
initialize $A = \{v,w\}$\;
pick $a \lrarr b$ such that $a \in \an_\G(v)$ and $b \in\an_\G(w)$\;
initialize $W = \{a,b\} \cup \dis^v(a) \cup \dis^w(b)$\;
}{
initialize $W = A = \{v,w\}$\;
}
\While{$A \neq \emptyset$}{
    $A := \dec(W) \setminus W$\;
	$A := (A \cup \dis^{v}(A)) \setminus W$\;
	$W := W \dot\cup A$\;
}
\Return{$W$}
\caption{Find minimal set $W$ that allows $X_v$ and $X_w$ to be perfectly correlated.}
\label{alg:mini1}
\end{algorithm}

  \begin{figure*}
  \begin{center}
 \large
  \begin{tikzpicture}[rv/.style={circle, draw, thick, minimum size=6mm, inner sep=0.5mm}, node distance=17.5mm, >=stealth]
  \pgfsetarrows{latex-latex};
  \begin{scope}
  \node[rv] (w) {$w$};
  \node[rv, above of=w, yshift=-2mm] (b) {$b$};
  \node[rv, right of=w, xshift=-2mm] (c) {$c$};
  \node[rv, below of=w, yshift=2mm] (v) {$v$};
  \node[rv, left of=w] (a) {$a$};
  \node[rv, right of=v] (d) {$d$};
  \draw[<->, very thick, color=red] (b) to[bend right=20] (c);
  \draw[<->, very thick, color=red] (b) to[bend left=60] (d);
    \draw[<->, very thick, color=red] (c) to (d);
  \draw[<->, very thick, color=red] (b) to (a);
    \draw[<->, very thick, color=red] (a) to[bend right] (v);
  \draw[->, very thick, color=blue] (a) -- (v);
  \draw[->, very thick, color=blue] (c) -- (v);
  \draw[->, very thick, color=blue] (d) -- (v);
  \draw[->, very thick, color=blue] (a) -- (w);
  \draw[->, very thick, color=blue] (b) -- (c);
  \draw[->, very thick, color=blue] (w) -- (c);
    \node[below of=v, yshift=7.5mm] {(a)};
    \end{scope}
      \begin{scope}[xshift=5.5cm]
  \node[rv] (w) {$w$};
  \node[rv, above of=w, yshift=-2mm] (b) {$b$};
  \node[rv, right of=w, xshift=-2mm] (c) {$c$};
  \node[rv, below of=w, yshift=2mm] (v) {$v$};
  \node[rv, left of=w] (a) {$a$};
  \node[rv, right of=v] (d) {$d$};
  \draw[<->, very thick, color=red] (c) to (d);
  \draw[<->, very thick, color=red] (b) to (a);
    \draw[<->, very thick, color=red] (a) to[bend left=20] (c);
  \draw[<->, very thick, color=red] (b) to[bend left=60] (d);

  \draw[->, very thick, color=blue] (a) -- (v);
  \draw[->, very thick, color=blue] (b) to[bend right] (v);
  \draw[->, very thick, color=blue] (c) -- (v);
  \draw[->, very thick, color=blue] (d) -- (v);
  \draw[->, very thick, color=blue] (a) -- (w);
  \draw[->, very thick, color=blue] (b) -- (c);
  \draw[->, very thick, color=blue] (w) -- (c);
  \draw[->, very thick, color=blue] (w) -- (v);
    \node[below of=v, yshift=7.5mm] {(b)};
    \end{scope}
          \begin{scope}[xshift=11cm]
  \node[rv] (w) {$w$};
  \node[rv, above of=w, yshift=-2mm] (b) {$b$};
  \node[rv, right of=w, xshift=-2mm] (c) {$c$};
  \node[rv, below of=w, yshift=2mm] (v) {$v$};
  \node[rv, right of=v] (d) {$d$};
  \node[rv, left of=w] (a) {$a$};
    \node[rv, red, below of=a, xshift=2mm, yshift=2mm] (U2) {\normalsize $h_2$};
    \node[rv, red, above of=a, xshift=2mm, yshift=-2mm] (U1) {\normalsize $h_1$};
    \node[rv, red, below of=c, xshift=1cm, yshift=8.75mm] (U4) {\normalsize $h_4$};
    \node[rv, red, above of=c, xshift=-4mm, yshift=-4mm] (U3) {\normalsize $h_3$};
    \node[rv, red, right of=b, xshift=8mm, yshift=2mm] (U5) {\normalsize $h_5$};
  \draw[->, very thick, color=red] (U2) -- (a);
  \draw[->, very thick, color=red] (U2) -- (v);
  \draw[->, very thick, color=red] (U1) -- (a);
  \draw[->, very thick, color=red] (U1) -- (b);
  \draw[->, very thick, color=red] (U3) -- (b);
  \draw[->, very thick, color=red] (U3) -- (c);
  \draw[->, very thick, color=red] (U4) -- (d);
  \draw[->, very thick, color=red] (U4) -- (c);
  \draw[->, very thick, color=red] (U5) -- (d);
  \draw[->, very thick, color=red] (U5) to[bend right=10] (b);
  \draw[->, very thick, color=blue] (a) -- (v);
  \draw[->, very thick, color=blue] (c) -- (v);
  \draw[->, very thick, color=blue] (d) -- (v);
  \draw[->, very thick, color=blue] (a) -- (w);
  \draw[->, very thick, color=blue] (b) -- (c);
  \draw[->, very thick, color=blue] (w) -- (c);
  \node[below of=v, yshift=7.5mm] {(c)};
    \end{scope}
              \begin{scope}[xshift=0cm, yshift=-5cm]
  \node[rv] (w) {$w$};
  \node[rv, above of=w, yshift=-4mm] (b) {$b$};
  \node[rv, right of=w] (c) {$c$};
  \node[rv, below of=w, yshift=4mm] (v) {$v$};
  \node[rv, left of=w] (a) {$a$};
  \node[rv, right of=v] (d) {$d$};
  \draw[<->, very thick, color=red] (b) to[bend left] (c);
    \draw[<->, very thick, color=red] (c) to (d);
  \draw[<->, very thick, color=red] (b) to (a);
    \draw[<->, very thick, color=red] (a) to[bend right] (v);
  \draw[->, very thick, color=blue] (a) -- (v);
  \draw[->, very thick, color=blue] (c) -- (v);
  \draw[->, very thick, color=blue] (d) -- (v);
  \draw[->, very thick, color=blue] (b) -- (c);
  \draw[->, very thick, color=blue] (w) -- (c);
  \node[below of=v, yshift=7.5mm] {(d)};
    \end{scope}
                  \begin{scope}[xshift=5.5cm, yshift=-5cm]
  \node[rv] (w) {$w$};
  \node[rv, above of=w, yshift=-4mm] (b) {$b$};
  \node[rv, right of=w] (c) {$c$};
  \node[rv, below of=w, yshift=4mm] (v) {$v$};
  \node[rv, left of=w] (a) {$a$};
  \draw[<->, very thick, color=red] (b) to[bend left] (c);
  \draw[<->, very thick, color=red] (b) to (a);
    \draw[<->, very thick, color=red] (a) to[bend right] (v);
  \draw[->, very thick, color=blue] (a) -- (v);
  \draw[->, very thick, color=blue] (c) -- (v);
  \draw[->, very thick, color=blue] (b) -- (c);
  \draw[->, very thick, color=blue] (w) -- (c);
  \node[below of=v, yshift=7.5mm] {(e)};
    \end{scope}
                     \begin{scope}[xshift=11cm, yshift=-5cm]
   \node[rv] (w) {$w$};
  \node[rv, above of=w, yshift=-4mm] (b) {$b$};
  \node[rv, right of=w] (c) {$c$};
  \node[rv, below of=w, yshift=4mm] (v) {$v$};
  \node[rv, left of=w] (a) {$a$};
    \node[rv, red, below of=a, xshift=4mm, yshift=4mm] (U2) {\normalsize $h_2$};
    \node[rv, red, above of=a, xshift=4mm, yshift=-4mm] (U1) {\normalsize $h_1$};
    \node[rv, red, above of=c, xshift=-4mm, yshift=-4mm] (U3) {\normalsize $h_3$};
  \draw[->, very thick, color=red] (U2) -- (a);
  \draw[->, very thick, color=red] (U2) -- (v);
  \draw[->, very thick, color=red] (U1) -- (a);
  \draw[->, very thick, color=red] (U1) -- (b);
  \draw[->, very thick, color=red] (U3) -- (b);
  \draw[->, very thick, color=red] (U3) -- (c);
  \draw[->, very thick, color=blue] (a) -- (v);
  \draw[->, very thick, color=blue] (c) -- (v);
  \draw[->, very thick, color=blue] (b) -- (c);
  \draw[->, very thick, color=blue] (w) -- (c);
  \node[below of=v, yshift=7.5mm] {(f)};
    \end{scope}

  \end{tikzpicture}
 \normalsize
  \caption{An illustration of the construction in Proposition \ref{prop:dir}, for the vertices $v$ and $w$.  
  (a) An ADMG $\G$, (b) its MArG projection $\G^\dag$, and (c) its canonical DAG $\overline\G$.  In (d) we reduce
  the directed and bidirected edges from (a) to spanning trees, as well as removing 
  incoming edges to $z$; 
  in (e) we remove the vertex $d$ using Corollary \ref{cor:trees3}, and in (f) we give the 
  canonical DAG for the graph in (e).}
  \label{fig:exm0}
  \end{center}
  \end{figure*}
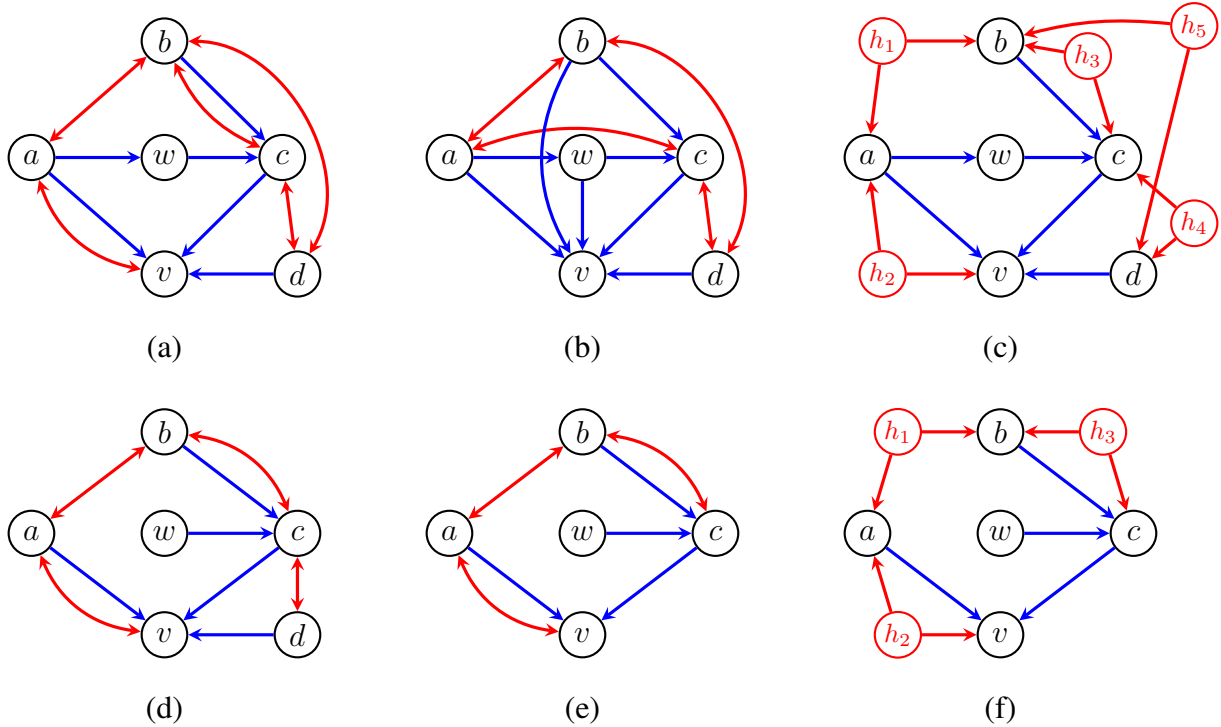

  \begin{figure*}
  \begin{center}
 \large
  \begin{tikzpicture}[rv/.style={circle, draw, thick, minimum size=6mm, inner sep=0.5mm}, node distance=17.5mm, >=stealth]
  \pgfsetarrows{latex-latex};
  \begin{scope}
  \node[rv] (v) {$v$};
  \node[rv, above of=v, yshift=-4mm] (1) {$a$};
  \node[rv, right of=v] (2) {$b$};
  \node[rv, below of=v, yshift=4mm] (w) {$w$};
  \node[rv, left of=v] (3) {$c$};
  \draw[<->, very thick, color=red] (3) -- (w);
  \draw[<->, very thick, color=red] (1) -- (2);
  \draw[<->, very thick, color=red] (v) to (2);
  \draw[<->, very thick, color=red] (1) to[bend right=30] (w);
    \draw[<->, very thick, color=red] (2) to[bend right=-30] (w);
  \draw[->, very thick, color=blue] (1) -- (v);
  \draw[->, very thick, color=blue] (3) -- (v);
  \draw[->, very thick, color=blue] (2) -- (w);
    \node[below of=w, yshift=7.5mm] {(a)};
    \end{scope}
      \begin{scope}[xshift=5cm]
  \node[rv] (v) {$v$};
  \node[rv, above of=v, yshift=-4mm] (1) {$a$};
  \node[rv, right of=v] (2) {$b$};
  \node[rv, below of=v, yshift=4mm] (w) {$w$};
  \node[rv, left of=v] (3) {$c$};
  \draw[<->, very thick, color=red] (3) -- (w);
  \draw[<->, very thick, color=red] (1) -- (2);
  \draw[<->, very thick, color=red] (v) to (2);
  \draw[<->, very thick, color=red] (1) to[bend right=30] (w);
  \draw[->, very thick, color=blue] (1) -- (v);
  \draw[->, very thick, color=blue] (3) -- (v);
  \draw[->, very thick, color=blue] (2) -- (w);
  \node[below of=w, yshift=7.5mm] {(b)};
    \end{scope}
      \begin{scope}[xshift=10cm, yshift=0cm]
   \node[rv] (v) {$v$};
  \node[rv, above of=v, yshift=-4mm] (a) {$a$};
  \node[rv, right of=v] (b) {$b$};
  \node[rv, below of=v, yshift=4mm] (w) {$w$};
  \node[rv, red, left of=v, xshift=4mm] (U2) {\normalsize $h_2$};
    \node[rv, red, right of=v, xshift=-8.75mm, yshift=-2mm] (U3) {\normalsize $h_3$};
    \node[rv, red, above of=b, xshift=-4mm, yshift=-4mm] (U1) {\normalsize $h_1$};
  \draw[->, very thick, color=red] (U2) -- (a);
  \draw[->, very thick, color=red] (U2) -- (w);
  \draw[->, very thick, color=red] (U1) -- (a);
  \draw[->, very thick, color=red] (U1) -- (b);
  \draw[->, very thick, color=red] (U3) -- (b);
  \draw[->, very thick, color=red] (U3) -- (v);
  \draw[->, very thick, color=blue] (a) -- (v);
  \draw[->, very thick, color=blue] (b) -- (w);
%
  \node[below of=w, yshift=7.5mm] {(c)};
    \end{scope}
  \end{tikzpicture}
 \normalsize
  \caption{An illustration of the construction in Proposition \ref{prop:bi}, for the vertices $v$ and $w$.  
  (a) An ADMG; (b) here we reduce the bidirected edges 
  to a spanning tree; (c) we apply Corollary \ref{cor:trees2} to remove $c$ and take
  the canonical DAG.}
  \label{fig:exm1}
  \end{center}
  \end{figure*}
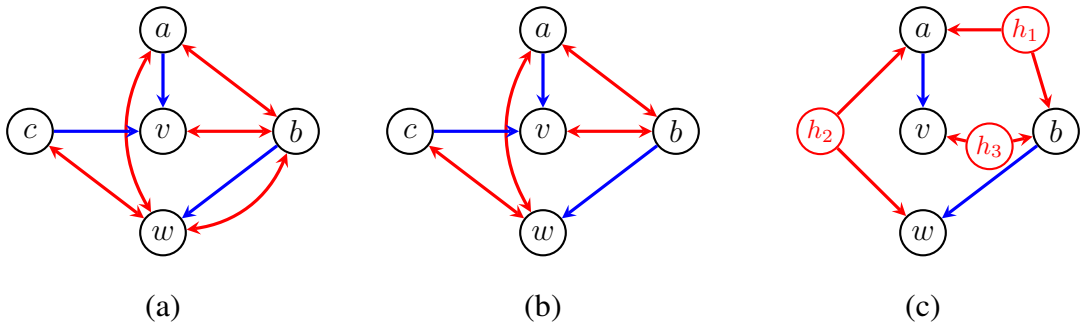

  \begin{figure*}
  \begin{center}
 \large
  \begin{tikzpicture}[rv/.style={circle, draw, thick, minimum size=6mm, inner sep=.5mm}, node distance=18mm, >=stealth]
  \pgfsetarrows{latex-latex};
  \begin{scope}
  \node[rv] (1) {$a$};
  \node[rv, right of=1, xshift=-3mm] (2) {$b$};
  \node[rv, above right of=2, yshift=0mm] (3) {$c$};
  \node[rv, below right of=2, yshift=0mm] (4) {$d$};
  \draw[<->, very thick, color=red] (3) to[bend right=30] (2);
  \draw[<->, very thick, color=red] (2) to[bend right=30] (4);
  \draw[->, very thick, color=blue] (1) -- (2);
  \draw[->, very thick, color=blue] (2) -- (3);
  \draw[->, very thick, color=blue] (2) -- (4);
    \node[below of=2, yshift=5mm, xshift=-10mm] {(a)};
    \end{scope}
      \begin{scope}[xshift=5.5cm]
  \node[rv] (1) {$a$};
  \node[rv, right of=1, xshift=-3mm] (2) {$b$};
  \node[rv, above right of=2, yshift=0mm] (3) {$c$};
  \node[rv, below right of=2, yshift=0mm] (4) {$d$};
  \draw[<->, very thick, color=red] (3) -- (4);
  \draw[->, very thick, color=blue] (1) -- (3);
  \draw[->, very thick, color=blue] (1) -- (4);
  \draw[->, very thick, color=blue] (2) -- (3);
  \draw[->, very thick, color=blue] (1) -- (2);
  \draw[->, very thick, color=blue] (2) -- (3);
  \draw[->, very thick, color=blue] (2) -- (4);
    \node[below of=2, yshift=5mm, xshift=-10mm] {(b)};
    \end{scope}
  \end{tikzpicture}
 \normalsize
  \caption{(a) A graph in which it is not possible to have $X_a = X_c = X_d$ 
  almost surely and independent of $X_b$, even though this clearly is possible 
  in its arid projection (b).}
  \label{fig:counter_exm}
  \end{center}
  \end{figure*}
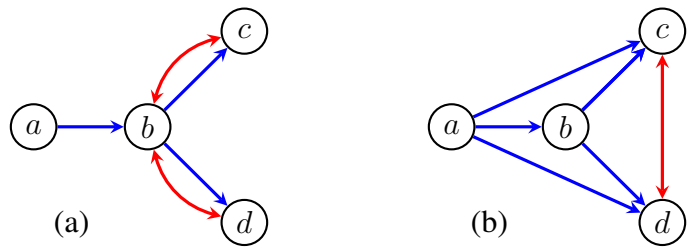

\begin{prop}
The complexity of Algorithm \ref{alg:mini1}  is $O(|V|)$, where
$V$ is the vertex set of $\G$.
\end{prop}

\begin{proof}
Note that, since the graph consists of two trees, there are only
$2(|V| - 1)$ edges in total.  
Suppose we have the children and spouses of each vertex (otherwise
we may need to run an $O(|V|)$ algorithm to obtain these, e.g.\ from 
the parent sets). 
Finding the district in a particular direction may involve rearranging
the spouses so that the one towards $v$ (or $w$) is listed first.  Again,
this can be done in $O(|V|)$ time.

If the `if' condition is satisfied, then we must find a bidirected edge 
that is an ancestor of both $v$ and $w$.  These two sets can be determined 
in linear time, and the check for an edge will also be linear.
If the `else' condition is satisfied, then there is essentially nothing to do. 

Then we just run the algorithm
to find the descendants of $w$, or of $W$ in the `if' case and 
then take the first entry 
in each list of spouses to find bidirected paths to $v$.
Since we ignore any vertices that we have seen before (or stop at this 
point), this entire
procedure will be worst case linear in the number of vertices.
\end{proof}


\section{Examples} \label{sec:exm}

\begin{exm}
Consider the graph in Figure \ref{fig:exm0}(a), and consider the pair
$v,w$.  Note that, in the arid projection of this graph (see (b)), there is an 
edge $w \to v$; hence, by Theorem \ref{thm:main}, we know it must 
is possible to set up a distribution, Markov with respect to the 
canonical DAG $\overline\G$ in \ref{fig:exm0}(c), such that $X_v = X_w$ and 
these are independent of all the other observed variables.  

To do this, we simply apply Propositions \ref{prop:trees} and Corollary \ref{cor:trees3} 
to obtain the graph in (e), and then follow the rules stated in the 
proof of Proposition \ref{prop:dir}.  We select a value $X_w=x_w$, sample independent 
Bernoulli random variables for each of $H_1$, $H_2$ and $H_3$, and then let all 
other variables be xor-gates of their parents.  We therefore have:
\begin{align*}
X_a &= H_1 \oplus H_2
\qquad\qquad X_b = H_1 \oplus H_3\\
X_c &= H_3 \oplus X_b \oplus x_w = H_1 \oplus x_w\\
X_v &= H_2 \oplus X_a \oplus X_c,
\end{align*}
and note that by substituting in, we find: 
\begin{align*}
X_v &= H_2 \oplus (H_1 \oplus H_2) \oplus (H_1 \oplus x_w) = x_w.
\end{align*}
We also note that Lemma \ref{lem:sumtozero2} implies that 
all the other variables are independent.

\end{exm}

\begin{exm}
Consider the graph in Figure \ref{fig:exm1}(a); we can reduce this to 
the graph in (c) using Propositions \ref{prop:trees} and \ref{prop:complete}.  
Then if we set the latent variables to be $H_1$, $H_2$, $H_3$ (as indicated in 
Figure \ref{fig:exm1}(c)), we have:
\begin{align*}
X_a &= H_1 \oplus H_2 \qquad\qquad X_b = H_1 \oplus H_3\\
X_v &= X_a \oplus H_3 = H_1 \oplus H_2 \oplus H_3\\
\text{and} \quad 
X_w &= H_2 \oplus X_b = H_2 \oplus H_1 \oplus H_3;
\end{align*}
hence, we indeed have $X_v = X_w$.  In addition, we clearly have joint independence 
just by applying Lemma \ref{lem:sumtozero2}. 
 
Note that, if we had removed the $a \lrarr w$ edge when obtaining a spanning 
tree, we would still have 
been able to obtain the result.  In this case everything reduces to essentially
be the same as the IV model: we can remove $a$ as well as 
$c$, and we have $X_v = H_3$ and (if the new latent variable above $X_b$ 
and $X_w$ is $H_4$) 
we have $X_w = X_b \oplus H_4 = (H_3 \oplus H_4) \oplus H_4 = H_3$. 
By similar reasoning to before, the other variables $X_a, X_c$ are 
jointly independent of each other and $X_v$.
\end{exm}

\begin{rmk}
Note that the case where the arid graph induces a directed edge 
allows us to choose the value that the two variables take by intervening
on the parent variable.  However, in the case where a bidirected edge is
introduced it is just a function of some hidden variables that we
may not be able to control.
\end{rmk}

\section{Discussion} \label{sec:discuss}

This paper demonstrates that DAG models with hidden variables
may display some surprising and counterintuitive relationships
between the variables in them.  
In terms of causal model search, these results suggest that we 
may be better off restricting all ADMGs to only maximal arid graphs, 
since their 
counterparts which are not maximal and arid may exhibit spurious
dependencies that do not reflect the overall structure.  Note also
that arid graphs are precisely those that are globally identifiable
under the assumption of linear structural equations \citep{drton11global}.

The results also provide further justification for the nested 
Markov model, since they show that
the nested model precisely characterizes which pairs of 
variables can---and cannot---be made to be perfectly dependent. 
In other words, either two vertices possess a nested Markov constraint
between them, or they can be made to be exactly equal and independent
of all other variables.  It also gives fast methods 
for checking whether certain distributions are compatible with 
a given hidden variable model.  

\subsection{Extension to multiple variables}

An obvious extension to this work is to consider 
the possibility of simultaneously setting all $(X_a : a \in A)$ 
for some set $A \subseteq V$ to be equal, with joint independence of all 
other variables.  

The graph in Figure \ref{fig:counter_exm}(a) shows that the results 
for pairwise equality do not extend in a simple way to larger sets.  
Consider 
the triple $\{a,c,d\}$, for example.  In the MArG projection 
of $\G$ (see Figure \ref{fig:counter_exm}(b)) it clearly \emph{is} 
possible to set these three variables
to be equal, since $c,d$ are both children of $a$.  However, in
fact this is not possible using the canonical DAG for $\G$, 
which can be verified by applying the results of \citet{wolfe19inflation}
or \citet{fraser20combinatorial}
(Elie Wolfe, personal communication.)

\textbf{Acknowledgements.}  We thank Elie Wolfe for confirming that 
the graph in Figure \ref{fig:counter_exm}(a) cannot represent the 
three-way perfect dependence.

\bibliographystyle{abbrvnat}
\bibliography{refs.bib}

\appendix

\section{Probabilistic Result} \label{sec:proofs}

\begin{lem} \label{lem:sumtozero2}
Consider a collection of variables $X_1, \ldots, X_k$, where 
each $X_i$ is a sum modulo 2 of some subset of its predecessors, 
and a number of independent Bernoulli random variables with 
parameter $1/2$.  

The Bernoulli 
random variables are each added to exactly two of the $X_i$s, and 
the edges induced by this graph form a tree over $\{1,\ldots,k\}$.

Then any subset of the random variables $X_1, \ldots, X_{k-1}$ contains
mutually (and jointly) independent $\operatorname{Bernoulli}(1/2)$ 
variables, but 
\[
\sum_{i=1}^k X_i = 0 \; \operatorname{mod} 2.
\]
\end{lem}

\begin{proof}
We proceed by induction.  $X_1 = \sum_{j \in T_1} Z_j \mod 2$ for
some non-empty set $T_1$ of independent Bernoulli random variables, 
so $X_1$ itself must be a $\operatorname{Bernoulli}(1/2)$ r.v.  
Now, suppose that $X_1, \ldots, X_{\ell-1}$ are independent 
$\operatorname{Bernoulli}(1/2)$ r.v.s, and consider $X_\ell$ for $\ell < k$.  
Since the random variables form a tree, then
\begin{align*}
X_\ell = \sum_{i \in S_\ell} X_i + \sum_{i \in T_\ell} Z_i \mod 2
\end{align*}
is independent of $X_1, \ldots, X_{\ell-1}$ if and only if 
\begin{align*}
X_\ell^* := \sum_{i \in T_\ell} Z_i \mod 2
\end{align*}
is independent of $X_1^*, \ldots, X_{\ell-1}^*$.  Note, however, that 
since the $Z_i$s form a tree, there must be at least one of these 
variables that includes a $Z_i$ not seen elsewhere, so we know that 
this variable is independent of the rest.  When we remove this variable, we then 
have a sub-tree, so we can repeat the argument 
until we have shown that all the variables are independent.

When we get to $X_k$ however, note that we have an invertible 
transformation between the $X_i$s and the $Z_i$s, and since the
$Z_i$s form a tree over the $X_i$s, there is one fewer of them.
It follows that the full collection of $X_i$s is dependent.
\end{proof}

\newpage

\section{Graphs} \label{sec:graphs}

The first concept we will need is an extension to ADMGs in which we 
allow some vertices to be `fixed'.  
We define the \emph{siblings} of a vertex to be its neighbours via 
bidirected edges:
\begin{align*}
\sib_\G(v) &\equiv \{ w : v \lrarr w \text{ in }\G \}.
\end{align*}

A CADMG $\G(V,W)$ is an ADMG with a set of \emph{random} vertices $V$ and 
\emph{fixed} vertices $W$, with the property that 
$\sib_\G(w) \cup \pa_\G(w) = \emptyset$ for every $w \in W$.  
An example can be found in Figure \ref{fig:verma}(b); note
that we depict fixed vertices with rectangular nodes, and random
vertices with round nodes.  Random vertices in a CADMG correspond 
to random variables, as in standard graphical models, while fixed 
vertices correspond to variables that were fixed 
to a specific value by some operation, such as conditioning or
causal interventions.  The genealogical relations in Section \ref{sec:prelim} 
generalize in a straightforward way to CADMGs by ignoring the distinction between
$V$ and $W$; the only exception is that districts are only defined for 
random vertices, so that the districts in the graph partition only
$V$, rather than $V \cup W$.

\subsection{Latent Projection} \label{sec:lp}

The \emph{latent projection} of a CADMG $\G(V \dot\cup L, W)$ to 
another graph $\G'(V, W)$ is given by following the rules:
\begin{itemize}
\item if there is a directed path from $a \in V \cup W$ to 
$b \in V$, and any interior vertices are in $L$, then 
add $a \rightarrow b$;
\item if there is a path between $a,b\in V$ without any adjacent 
arrowheads, and any interior vertices are in $L$, that
starts and ends with an arrow at $a$ and $b$, then add $a \lrarr b$.
\end{itemize}

As an example, consider the ADMG in Figure \ref{fig:lp}(a), with 
variable $h$ designated as latent.  Then the projection of this is
given by the ADMG in Figure \ref{fig:lp}(b).

  \begin{figure}
  \begin{center}
 \large
  \begin{tikzpicture}[rv/.style={circle, draw, thick, minimum size=6mm, inner sep=0.5mm}, node distance=18mm, >=stealth]
  \pgfsetarrows{latex-latex};
  \begin{scope}
  \node[rv, red] (0) {$h$};
  \node[rv, above left of=0] (1) {$a$};
    \node[rv, above right of=0] (2) {$b$};
  \node[rv, below left of=0] (3) {$c$};
  \node[rv, below right of=0] (4) {$d$};
  \draw[<->, very thick, color=red] (3) to[bend right=-20] (0);
  \draw[<->, very thick, color=red] (0) to[bend right=-20] (4);
  \draw[->, very thick, color=blue] (1) -- (0);
  \draw[->, very thick, color=blue] (0) -- (3);
  \draw[->, very thick, color=blue] (0) -- (4);
  \draw[->, very thick, color=blue] (2) -- (4);
  \draw[<->, very thick, color=red] (2) -- (0);
    \node[below of=0, xshift=0mm, yshift=-3mm] {(a)};
    \end{scope}
      \begin{scope}[xshift=4cm, yshift=8mm]
  \node[rv] (1) {$a$};
  \node[rv, right of=1] (2) {$b$};
  \node[rv, below of=1] (3) {$c$};
  \node[rv, below of=2] (4) {$d$};
  \draw[<->, very thick, color=red] (3) -- (4);
  \draw[->, very thick, color=blue] (1) -- (3);
  \draw[->, very thick, color=blue] (1) -- (4);
  \draw[<->, very thick, color=red] (2) -- (3);
  \draw[->, very thick, color=blue] (2) -- (4);
  \draw[<->, very thick, color=red] (2) to[bend left] (4);
    \node[below of=3, yshift=6.5mm, xshift=9mm] {(b)};
    \end{scope}
  \end{tikzpicture}
 \normalsize
  \caption{(a) An ADMG in which $h$ is latent; (b) its latent projection over $\{a,b,c,d\}$.}
  \label{fig:lp}
  \end{center}
  \end{figure}
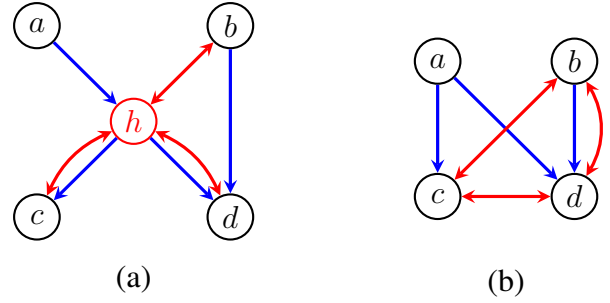
  
\subsection{Arid Projection} \label{sec:aridproj}

\begin{exm}
The \emph{maximal arid projection} of the ADMG $\G$ in Figure \ref{fig:aridproj}(a)
is given in \ref{fig:aridproj}(b).  In the graph (a) we have
$\langle d \rangle_\G = \{b,d\}$, so $\pa_\G(\langle d \rangle) = \{a,b,c\}$.
As a result, in (b) all these vertices are parents of $d$.
In addition, $\langle \{d,e\} \rangle_\G = \{b,c,d,e\}$ which is bidirected connected,
so we add the edge $d \lrarr e$ into (b).  All other adjacencies are as in (a).

 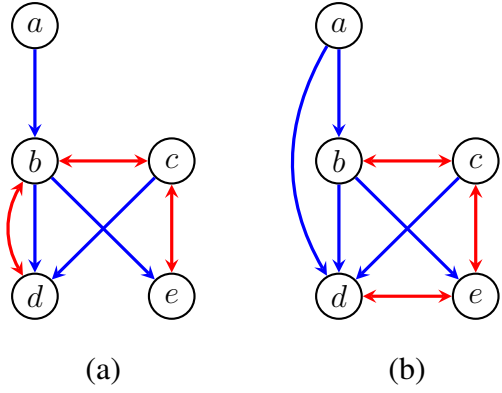
\begin{figure}
  \begin{center}
 \large
  \begin{tikzpicture}[rv/.style={circle, draw, thick, minimum size=6mm, inner sep=0.5mm}, node distance=18mm, >=stealth]
  \pgfsetarrows{latex-latex};
  \begin{scope}
  \node[rv] (0) {$a$};
  \node[rv, below of=0] (1) {$b$};
  \node[rv, right of=1] (2) {$c$};
  \node[rv, below of=1] (3) {$d$};
  \node[rv, below of=2] (4) {$e$};
  \draw[<->, very thick, color=red] (1) -- (2);
  \draw[<->, very thick, color=red] (2) -- (4);
  \draw[->, very thick, color=blue] (1) -- (3);
  \draw[->, very thick, color=blue] (1) -- (4);
  \draw[->, very thick, color=blue] (2) -- (3);
  \draw[->, very thick, color=blue] (0) -- (1);
  \draw[<->, very thick, color=red] (1) to[bend right] (3);
    \node[below of=3, xshift=9mm, yshift=8mm] {(a)};
    \end{scope}
      \begin{scope}[xshift=4cm]
  \node[rv] (0) {$a$};
  \node[rv, below of=0] (1) {$b$};
  \node[rv, right of=1] (2) {$c$};
  \node[rv, below of=1] (3) {$d$};
  \node[rv, below of=2] (4) {$e$};
  \draw[<->, very thick, color=red] (1) -- (2);
  \draw[<->, very thick, color=red] (2) -- (4);
  \draw[<->, very thick, color=red] (3) -- (4);
  \draw[->, very thick, color=blue] (1) -- (3);
  \draw[->, very thick, color=blue] (1) -- (4);
  \draw[->, very thick, color=blue] (2) -- (3);
  \draw[->, very thick, color=blue] (0) -- (1);
  \draw[->, very thick, color=blue] (0) to[bend right] (3);
    \node[below of=3, xshift=9mm, yshift=8mm] {(b)};
    \end{scope}
  \end{tikzpicture}
 \normalsize
  \caption{(a) An ADMG $\G$ which is neither maximal nor arid; (b) its maximal arid projection.}
  \label{fig:aridproj}
  \end{center}
  \end{figure}
\end{exm}  

\section{The Nested Markov Model} \label{sec:nestedmodel}


\subsection{Fixing}

A vertex $r \in V$ is said to be \emph{fixable} in a CADMG $\G(V,W)$ if 
$\dis_\G(r) \cap \dec_\G(r) = \emptyset$.  For instance, the vertices $a$,
$c$ and $d$ are all fixable in the graph in Figure \ref{fig:verma}(a), 
but $b$ is not because $d$ is both its descendant and its sibling.  

For any $v \in V$, such that $\ch_\G(v) = \emptyset$,
the \emph{Markov blanket} of $v$ in a CADMG $\G$ is
defined as
\begin{align*}
\mbl_\G(v) \equiv (\dis_\G(v) \cup \pa_\G(\dis_\G(v))) \setminus \{v\};
\end{align*}
that is, the set of vertices that are connected to $v$ by paths
with an arrow at $v$ and two arrowheads at each internal vertex. 
We can generalize this definition to any vertex that is childless 
within its own district.

Given a CADMG $\G(V,W)$, and a fixable $r \in V$, the fixing operation 
$\phi_r(\G)$ yields a new CADMG $\widetilde\G(V\setminus\{r\},W\cup\{r\})$ 
obtained from $\G(V,W)$ by removing all edges of the form $\to r$ and 
$\lrarr r$, and keeping all other edges. 
Given a kernel $q_V(x_V \cmid x_W)$ associated with a CADMG $\G(V,W)$, 
and a fixable $r \in V$, the fixing operation
$\phi_r(q_V; \G)$ yields a new kernel
\[
\tilde{q}_{V\setminus\{r\}}(x_{V\setminus\{r\}} \cmid x_{W}, x_r) \equiv \frac{q_V(x_V \cmid x_W)}{q_V(x_r \cmid x_{\mbl_\G(r)})}.
\]
A result in \cite{richardson:17} allows us to unambiguously define 
\begin{align*}
\phi_{R}(\G) &\equiv \phi_{r_k}( \ldots \phi_{r_2}(\phi_{r_1}(\G)) \ldots ),
\end{align*}
and similarly the kernel $\phi_R(p; \G)$ for distributions that are nested Markov
with respect to $\G$ (defined below).  Consequently, we just use sets to index 
fixings from now on.


If a fixing sequence exists for a set $R \subseteq V$ in $\G(V,W)$, 
we say $V \setminus R$ is a \emph{reachable set}.  
Such a set is called \emph{intrinsic} if the vertices in $V\setminus R$ are
bidirected-connected (so that $\phi_{V \setminus R}(\G)$ has only a single
district); this definition is equivalent to the definition in the main paper.  
We denote the collections of reachable and intrinsic sets in $\G$ 
respectively by ${\cal R}(\G)$ and ${\cal I}(\G)$.




For a CADMG $\G(V,W)$,
a (reachable) subset ${ C} \subseteq { V}$
is called a \emph{reachable closure} for ${ S} \subseteq { C}$
if the set of fixable vertices in $\phi_{V \setminus C}(\G)$ is a subset of ${S}$.
Every set ${S}$ in $\G$ has a unique reachable closure, which we denote
$\langle S \rangle_\G$ \citep{shpitser:18}.  Note that this set is generally 
a subset of what we earlier called the `closure'.

\subsection{Nested Markov Model}

We are now ready to define the nested Markov model $\M_n(\G)$.
Given an ADMG $\G$, we say that a distribution $p$ 
obeys the \emph{nested Markov property} with respect to $\G$ if 
for any reachable set $R$, we have that
$\phi_{V \setminus R}(p; \G)$ factorizes into kernels as
\begin{align*}
\phi_{V \setminus R}(p; \G) &= \prod_{D \in \mathcal{D}(\phi_{V \setminus R}(\G))} \phi_{V \setminus D}(p; \G).
\end{align*}
In other words, for any reachable graph, the associated
kernel factorizes into a product of the districts in that 
graph conditional on the parents of those districts. 

Note that this also means that $\phi_{V \setminus R}(p; \G)$ 
will be Markov with respect to the CADMG $\phi_{V \setminus R}(\G)$
for each reachable set $R$; see \citet{richardson03markov} for more
details on this.

\begin{exm}
Consider the ADMG in Figure \ref{fig:verma}(a).  The 
vertices $a$, $c$ and $d$ all satisfy the condition 
of being fixable, but $b$ does not since 
$d$ is both a descendant of, and in the same district as,
$b$.  The CADMG $\G(\{b,d\},\{a,c\})$ obtained after fixing $a$
and $c$ is shown in Figure \ref{fig:verma}(b).  Notice 
that fixing $c$ removes the edge $b \to c$, but
that the edge $c \to d$ is preserved. 
Applying m-separation to the graph shown in Figure \ref{fig:verma}(b),
yields
\begin{align*}
&X_d \indep X_a \mid X_c \quad [\phi_{ac}(p(x_{abcd});\G)].
\end{align*}
In addition, one can see easily that if an edge 
$a \to d$ had been present in the original 
graph, 
then we would not have obtained this
m-separation.

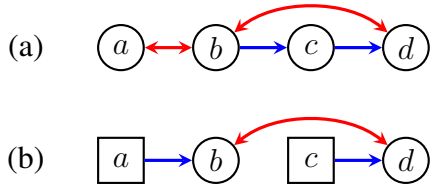
\begin{figure}
\begin{center}
\begin{tikzpicture}
[rv/.style={circle, draw, thick, minimum size=6mm, inner sep=0.5mm}, 
fv/.style={rectangle, draw, thick, minimum size=6mm, inner sep=0.5mm}, 
node distance=12.5mm, >=stealth]
\pgfsetarrows{latex-latex};
\large
\begin{scope}
 \node[rv] (1) {$a$};
 \node[rv, right of=1] (2) {$b$};
 \node[rv, right of=2] (3) {$c$};
 \node[rv, right of=3] (4) {$d$};
 \draw[<->, very thick, color=red] (1) -- (2);
 \draw[->, very thick, color=blue] (2) -- (3);
 \draw[->, very thick, color=blue] (3) -- (4);
\draw[<->, very thick, color=red] (2.40) .. controls +(40:.7) and +(140:.7) .. (4.140);
\node[left of=1, yshift=0mm, xshift=0mm] {(a)};
\end{scope}
\begin{scope}[yshift=-1.5cm]
 \node[fv] (1) {$a$};
 \node[rv, right of=1] (2) {$b$};
 \node[fv, right of=2] (3) {$c$};
 \node[rv, right of=3] (4) {$d$};
 \draw[->, very thick, color=blue] (1) -- (2);
 \draw[->, very thick, color=blue] (3) -- (4);
\draw[<->, very thick, color=red] (2.40) .. controls +(40:.7) and +(140:.7) .. (4.140);
\node[left of=1, yshift=0mm, xshift=0mm] {(b)};
\end{scope}
\normalsize
\end{tikzpicture}
\caption{(a) An ADMG ${\cal G}$ that is not ancestral; 
(b) a CADMG obtained from ${\cal G}$ in (a) by fixing $a$ and $c$.}
\label{fig:verma}
\end{center}
\end{figure}

\end{exm}

\subsection{Densely Connected Vertices}

Here we give a couple of slightly more detailed examples
than in the main text.

\begin{exm}
The vertices $a$ and $c$ in Figure \ref{fig:iv}(b) are `densely
connected', because they cannot be separated by any combination of 
conditioning or fixing, except by fixing $c$ (which just amounts to marginalizing 
it from the graph). 
Separately, for `gadget' graph in Figure \ref{fig:gadget}(b) the vertices
$c$ and $d$ are also `densely connected'.  Naturally, any pair of 
vertices joined by an edge is also densely connected.
\end{exm}

\section{Algorithms} \label{sec:algo}

\subsection{Spanning Tree}

Given a set $C$ and its subset of 
childless nodes $B$ (in our case this will
be either $\{v\}$ or $\{v,w\}$), pick a topological 
order on the vertices that places all elements of $B$ 
at the end.  Then, the last vertex before $B$ must be 
a parent of some element of $B$; pick the largest such 
element under the topological order.  

We then move backwards in the topological order, and 
each time a vertex has more than one child, we join it
to the vertex which has the shortest path to an element 
of $B$; if there is a tie, then we pick the largest 
element under the topological order.  This ensures that 
each vertex is joined to $B$ by the shortest possible directed 
path.

\subsection{Difficult Graphs} \label{sec:diff}

Consider the graph shown in Figure \ref{fig:comp}.  This 
can clearly be reduced to the graph $w \to v$, 
but the application of Proposition \ref{prop:complete} is
computationally difficult.  Note that no subset will work
apart from $\{z_1, \ldots, z_k\}$, and there are $3^k - 1$ 
possible sets to choose.

Algorithm 1 (with complexity proven to be $O(|V|)$) can be 
applied instead and will immediately return the graph $w \to v$.

\begin{figure}
\begin{center}
\begin{tikzpicture}
[rv/.style={circle, draw, thick, minimum size=6mm, inner sep=0.5mm}, node distance=12.5mm, >=stealth]
\pgfsetarrows{latex-latex};
\large
\begin{scope}
 \node[rv] (1) {$y_{1}$};
 \node[rv, below of=1] (1a) {$z_{1}$};
 \node[rv, right of=1] (2) {$y_{2}$};
 \node[rv, below of=2] (2a) {$z_{2}$};
 \node[rv, right of=2] (3) {$y_{3}$};
 \node[rv, below of=3] (3a) {$z_{3}$};
\node[right of=3, xshift=-3mm] (d1) {$\dots$};
\node[below of=d1] (d2) {$\dots$};
 \node[rv, right of=d1, xshift=-3mm] (k) {$y_{k}$};
 \node[rv, below of=k] (ka) {$z_{k}$};
 \node[rv, below of=3a] (v) {$v$};
  \node[rv, below right of=v] (w) {$w$};
   \draw[->, very thick, color=blue] (1) -- (1a);
 \draw[->, very thick, color=blue] (1a) -- (v);
 \draw[->, very thick, color=blue] (2) -- (2a);
 \draw[->, very thick, color=blue] (2a) -- (v);
 \draw[->, very thick, color=blue] (3) -- (3a);
 \draw[->, very thick, color=blue] (3a) -- (v);
 \draw[->, very thick, color=blue] (k) -- (ka);
 \draw[->, very thick, color=blue] (ka) -- (v);
 \draw[->, very thick, color=blue] (w) -- (v);
 \draw[<->, very thick, color=red] (1) -- (2);
 \draw[<->, very thick, color=red] (3) -- (2);
 \draw[->, very thick, color=red] (d1) -- (3);
 \draw[->, very thick, color=red] (d1) -- (k);
 \draw[<->, very thick, color=red] (1a) -- (2a);
 \draw[<->, very thick, color=red] (3a) -- (2a);
 \draw[->, very thick, color=red] (d2) -- (3a);
 \draw[->, very thick, color=red] (d2) -- (ka);
  \draw[<->, very thick, color=red] (k) to[bend left] (ka);
 \draw[<->, very thick, color=red] (1.215) .. controls +(215:1.5) and +(180:3) .. (v.180);
\end{scope}
\normalsize
\end{tikzpicture}
\caption{An ADMG for which a search for a set to satisfy Proposition 
\ref{prop:complete} is computationally difficult.}
\label{fig:comp}
\end{center}
\end{figure}
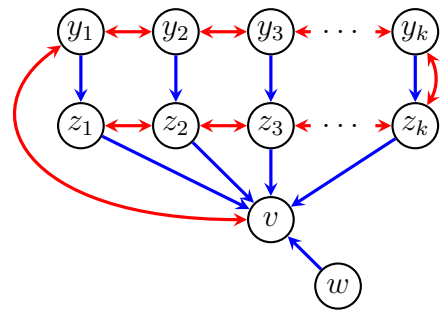

\end{document}